\documentclass[reqno,11pt]{amsart}
\usepackage{latexsym}
\usepackage{amsfonts}
\usepackage{amssymb}
\usepackage{mathrsfs}
\usepackage[all]{xy}
\usepackage{bbm}
\usepackage{xcolor}
\usepackage{ulem}
\usepackage[mathscr]{euscript}

\usepackage[english]{babel}
\usepackage{amsmath}
\usepackage{graphicx}
\usepackage[colorinlistoftodos]{todonotes}
\usepackage[colorlinks=true, allcolors=blue]{hyperref}
\usepackage{tikz-cd}
\usepackage{tikz}
\usepackage{pst-node}
\usepackage{mathtools}
\usetikzlibrary{graphs}
\usepackage{dsfont}
\usepackage{bbm}

\usepackage[T2A]{fontenc}

\theoremstyle{plain}
\newtheorem{thm}{Theorem}[section]
\newtheorem{claim}{Claim}[thm]
\newtheorem{lem}[thm]{Lemma}
\newtheorem{rem}[thm]{Remark}
\newtheorem{prop}[thm]{Proposition}
\newtheorem{cor}[thm]{Corollary}
\newtheorem{defn}[thm]{Definition}
\newtheorem{exmp}[thm]{Example}

\theoremstyle{definition}

\theoremstyle{remark}

\numberwithin{equation}{section}

\newcommand{\D}{\mathbf{Div}}
\newcommand{\SL}{\operatorname{SL}}
\newcommand{\ad}{\operatorname{ad}}

\newcommand{\Ad}{\operatorname{Ad}}

\newcommand{\R}{{\mathbb{R}}}

\newcommand{\Z}{{\mathbb{Z}}}
\newcommand{\N}{{\mathbb{N}}}

\newcommand{\va}{{\bf a}}
\newcommand{\vb}{{\bf b}}
\newcommand{\vc}{{\bf c}}

\newcommand{\ve}{{\bf e}}

\newcommand{\ignore}[1]{{}}

\newcommand {\comm}[1]   {\textcolor{red}{#1}}

\newif\ifignore
\newcommand{\lam}{\lambda}

\newcommand{\bR}{\mathbb{R}}
\newcommand{\bZ}{\mathbb{Z}}

\newcommand{\bT}{\mathbb{T}}

\newcommand\seta[1]{\left\{#1\right\}}
\newcommand\pa[1]{\left(#1\right)}
\newcommand\idist[1]{\langle#1\rangle}

\newcommand\on[1]{\operatorname{#1}}

\newcommand\tb[1]{\textbf{#1}}
\newcommand\mat[1]{\pa{\begin{matrix}#1\end{matrix}}}

\newcommand{\xtheta}{x_{\theta}}

\title[k-divergence]{K-Divergent lattices}
\author{Guy Lachman, Anurag Rao, Uri Shapira, Yuval Yifrach}
\address{{\tt guy.lachman@gmail.com}, 
{\tt oar\_garuna@hotmail.com}, {\tt ushapira@gmail.com}, {\tt huckhprj@me.com}}

\begin{document}

\begin{abstract}
    We introduce a novel concept in topological dynamics, referred to as $k$-divergence, which extends the notion of divergent orbits. Motivated by questions in the theory of inhomogeneous Diophantine approximations, we 
    investigate this notion in the dynamical system given by a certain flow on the space
    of unimodular lattices in $\bR^d$. Our main result is the existence of $k$-divergent lattices for any $k\ge 0$. In fact, we utilize the emerging theory of parametric geometry 
    of numbers and calculate the Hausdorff dimension of the set of $k$-divergent lattices.
\end{abstract}

%\subjclass{11J04; 11J13, 37A17, 37D40}
\maketitle
%\large
\section{Introduction}
\subsection{$k$-divergence}
In this article we investigate a novel concept in topological dynamics. We are motivated by a particular example of a diagonal flow acting on the space of lattices with specific implications to questions in Diophantine approximation. As a consequence, our analysis in this paper is restricted to this particular example. Nevertheless, we start the introduction abstractly.

Let $X$ be a topological space and $H$ a topological (semi) group acting continuously on $X$. The pair $(X,H)$ is referred to as a topological dynamical 
system. Given $x\in X$, the set 
\begin{equation*}
    Hx:=\{hx \in X :h\in H\} 
\end{equation*}
is called the \textit{orbit} of $x$. Broadly speaking, topological dynamics concerns itself with investigating the topological properties of orbits. A basic notion which is highly relevant to our discussion is that of divergence: 
An orbit $Hx$ is said to be \textit{divergent} if the map
$h\mapsto hx$ is a proper map from $H$ to $X$; that is, the preimage of any compact set in $X$ is compact in $H$. Clearly, it is of interest to discuss
divergent orbits only when both $H$ and $X$ are non-compact.
It will be convenient to use the following terminology\footnote{For the sake of simplicity all topological spaces are assumed to be metric with the Heine-Borel property: compact sets are those which are closed and bounded. Examples are manifolds equipped with a complete Riemannian metric.}: If $Y$ is a topological space, we say that a sequence $(y_n)_{n \in \N}\subset Y$ is \textit{divergent} and write $y_n\to \infty $ if $y_n$ eventually 
leaves any compact set in $Y$. Using this terminology an orbit $Hx$ is divergent if and only if 
\begin{equation*}
    h_n\to \infty  \implies h_nx \to \infty. 
\end{equation*} 

It is convenient to use the following terminology.
\begin{defn}[Asymptotic accumulation points]
    Given a topological dynamical system $(X,H)$ and a point $x\in X$, we have its set of {asymptotic accumulation points}:
\begin{equation*}
    D(x): = \{y\in X: \text{ There exists } (h_n)_{n \in \N} \subset H \text{ with } h_n\to  \infty \text{ and } h_nx \to y\}.
    \end{equation*}
\end{defn}
Since the acting (semi)-group $H$ is fixed we do not record it in the notation, although strictly speaking this set should have been denoted $D_H(x).$  For example, $Hx$ is divergent if and only if $D(x) = \varnothing$.
The following is the main concept which will interest us in the rest of this paper.
\begin{defn}[$k$-divergent orbits, or points]\label{def: k-div}
Given a point $x\in X$, a finite sequence 
\begin{equation*}
    (y_0:=x, y_1, y_2,\dots,y_k)\in X^{k+1}\ \text{ with }\ y_{i+1} \in D(y_i) \text { for } 0\leq i <k
\end{equation*}  
is called an {accumulation sequence of length $k$} for $x$. The point $x$ is said to be \textit{$k$-divergent} if $k$ is the supremum over all lengths of accumulation sequences for $x$. Finally, 
the orbit $Hx$ is said to be $k$-divergent if any of its points are such (note that this is well defined). 
\end{defn}
In the above definition, $k$ takes values in $\Z_{\geq 0} \cup \{\infty\}$.
Let 
$$\D(k): = \lbrace x\in X: \text{$x$ is $k$-divergent}\rbrace$$
so that these sets give a partition of $X$.
\begin{rem}
    
We note that from the statistical point of view the sets $\D(k)$ for $k$ finite are very small in the sense that they are null sets
with respect to {any} $H$-invariant probability measure. This is because: if $x$ is $k$-divergent
for finite $k$, then $x\notin D(x)$, and so $x$ is not a recurrent point.
{(cf. the proof of Theorem \ref{escape} below.)}
\end{rem}

As mentioned above, any point $x \in \D(0)$ only has accumulation sequences of length $0$ which implies that $Hx$ is divergent {(and conversely)}. 
Any point $x \in \D(1)$ has an accumulation sequence of length $1$ and none which are longer. This shows that it is not divergent but only has asymptotic accumulation points which are themselves divergent.
Indeed, Definition \ref{def: k-div} is recursive in nature and the reader can check that, for any $k \in \N$,
\begin{equation}\label{eq: D(k)-char}
\D(k) = \left\lbrace x\in X: D(x)\subset \cup_{j=0}^{k-1} \D(j)\right\rbrace\smallsetminus \cup_{j=0}^{k-1}\D(j).
\end{equation}
In words, $x$ is $k$-divergent if its asymptotic accumulation points are all $j$-divergent for $j<k$ and if $x$ itself is not $j$-divergent for any $j<k.$

\subsection{Results in the space of lattices}

Let $m$ and $n$ be positive integers and let $d=m+n$.
Let $G$ denote the Lie group $\SL_d(\R)$, let $\Gamma$ denote $\SL_d(\Z)$, and consider the noncompact homogeneous space $X_d:=G/\Gamma$. The space $X_d$ is the moduli space of lattices in $\R^d$ of covolume 1 (with $g\Gamma$ identified with the $\Z$-span of the columns of $g$). 
For the acting semi-group we take the one-parameter diagonal semi-group
$\{a_t \in G:t\in \R_{\geq 0}\}$ where 
\begin{equation}\label{at-standard}
    a_t := \text{diag}(e^{t/m}, \dots, e^{t/m}, e^{-t/n},\dots, e^{-t/n})
\end{equation}
(where the number of positive eigenvalues is $m$ and the number of negative eigenvalues is $n$). Henceforth, whenever we refer to Definition \ref{def: k-div} we keep $(X_d, (a_t)_{t\geq 0})$ as the underlying dynamical system.

Understanding {this} dynamical system is tightly related to Diophantine approximation via the so-called Dani correspondence. See for example \cite[Theorem 2.14]{Da} where divergent orbits under this flow are shown to be in correspondence with singular matrices. See \cite[Definition 2.13]{Da} for the definition of singular matrices or  see \eqref{eq: singular} where singular vectors are defined. 
Thus, the set $\D(0)$ has received a lot of attention from both the dynamical and number theoretic perspectives.
There are certain obvious lattices in $\D(0)$; the so called \textit{degenerate} divergent orbits which correspond to lattices having a rational subspace whose volume form goes to zero under $a_t$ (see \cite[Definition 2.8]{Da}). These lattices are easy to construct. On the other hand, the existence of nondegenerate divergent lattices was first demonstrated by Khintchine in 1926 \cite{Kh}. Much more recent work involves the computation of the exact Hausdorff dimension of $\D(0)$ (see \cite{C, CC, KKLM, DFSU, So}) which turns out to be 
$\dim X_d - \frac{nm}{n+m}$, and as this number is strictly bigger than the dimension of the degenerate divergent trajectories, one obtains a refinement of Khintchine's result showing the abundance of non-degenerate divergent lattices.

As we will explain shortly (see \S\ref{subsec:motivation}), our motivation {is} to prove existence of $k$-divergent lattices for 
$k>0$ finite. 
The situation seems to differ from the $\D(0)$ case because there are seem to be no lattices
which are $k$-divergent for obvious reasons. When $n=m=1$ and continued fractions are available, it is not hard to construct explicit examples of $k$-divergent lattices in the plane. 
In an unpublished work, the third named 
author was able to construct examples of $1$-divergent lattices building on the results in \cite{Sh2}. However, given the sophisticated theory of parametric geometry of numbers, one can establish existence in a far more systematic manner. 

This theory was developed in \cite{SS, Ro, DFSU, So} and guarantees existence of lattices whose orbit's behaviour is controlled  by certain combinatorial graphs called templates. The following, which is the main result 
of this paper, establishes that indeed, for any $k \in \N$, $\D(k)$ is nonempty. Moreover, we show that (unless $n=m=1$) their dimensions coincide with that of $\D(0)$. 
This result should be regarded as a demonstration of the strength and flexibility of the emerging tools from the theory of parametric geometry of numbers.
\begin{thm}\label{main-thm}
For any $k \ge 1$,   
\begin{equation}\label{eq: exact dimension}
   \dim (\D(k))   = \dim(X_d) - \frac{mn}{m+n}
\end{equation}
where $\dim$ refers to the Hausdorff dimension.
\end{thm}
Note, we prove Theorem~\ref{main-thm} by establishing tight upper and lower bounds on the dimension of $\D(k)$, for each $k \in \N$, with the lower bound coming from the formulas of \cite{DFSU} and the upper bound arising as an easy corollary of the main result in \cite{KKLM}. (See Corollaries~\ref{lower-bound-cor}, \ref{upper-bound-cor} below.)
\subsection{Inhomogeneous Diophantine approximation}\label{subsec:motivation}
This subsection is not needed for the rest of the paper and is included for motivational purposes and to put the discussion in context.
Our original motivation for proving the existence of $k$-divergent lattices for $k\ge 1$ stems from an attempt to construct a counterexample to a claim in \cite{Sh1}. In order to explain things in more detail we need to fix some terminology. Since this is a motivational discussion, we do not strive to describe things in utmost generality and restrict our attention to the case  $n=1$ 
 and $d=m+1$ so that
the sets $\D(k)$ are defined with respect to the semiflow 
\begin{equation*}
    a_t = \operatorname{diag}(e^{t/m}, \dots, e^{t/m}, e^{-t}) \text{ for } t \in \R_{\geq 0}.
\end{equation*}

One of the basic questions in inhomogeneous Diophantine approximation is to understand, for given vectors $\theta,\eta\in \bR^m$,  
the rate at which the quantity $\idist{q\theta - \eta}$ decays as $q\to \infty$ in the integers. Here, we have used the notation
$\langle \cdot \rangle$ for the distance to $\Z^m$ in the supremum norm. One set of particular interest is the set of \textit{badly approximable targets for} $\theta \in \R^m$: 
$$\on{Bad}_\theta :=\seta{\eta\in \bR^m : \liminf_{q\to\infty} q^{1/m}\idist{q\theta - \eta} > 0}.$$
Note that since $\idist{\cdot}$ is defined modulo $\bZ^m$, we can (and will) treat $\theta$ and $\eta$ as points on the $m$-torus $\bT^m: = \bR^m/\bZ^m$ and, as a result, treat $\on{Bad}_\theta$ as a subset of $\bT^m$. 

Recall that $\theta$ is said to be \textit{singular} if 
\begin{equation}\label{eq: singular}
    \lim_{Q\to\infty} Q^{1/m}\min\seta{ \idist{q\theta}:1\le q\le Q} = 0.
\end{equation} 
Dani's correspondence then says that if we let $\xtheta$ denote the $d$-dimensional lattice
$$\xtheta := \mat{I_m &\theta \\ 0 &1 }\bZ^d,$$
then $\theta$ is singular
if and only if $\seta{a_t\xtheta \in X :t\ge 0}$ is divergent. Alternatively, in our current terminology, $\theta$ is singular if and only if $\xtheta$ is $0$-divergent. We extend this terminology as follows:
\begin{defn}
    A vector $\theta\in \bR^m$ is said to be $k$-divergent if $x_\theta$ is a $k$-divergent lattice. 
\end{defn}
Let us note here that at the moment we do not have a reasonable characterization of $k$-divergence of 
$\theta$ in terms of Diophantine inequalities
for $k\ge 1$ but we expect that such a characterization should exist.

In \cite{ET} it was shown for every $\theta$, that $\on{Bad}_\theta$ has full Hausdorff dimension and is in fact a winning set for
Schmidt's game. The line of thought for the current discussion is to complement this fact about $\on{Bad}_\theta$ by establishing that 
although this set is big from the dimension point of view, it is small in the sense that it 
is a null set with respect to a large explicit family of probability measures on $\bT^d$.
Clearly, we have in mind some Diophantine restrictions on $\theta$:
For if $\theta \in \R^m$ is rational, then $\on{Bad}_\theta$ is {cofinite} and is consequently of full measure with respect to any non-atomic measure.
In \cite{Sh1} it was claimed that, once $\theta$ is non-singular, then $\on{Bad}_\theta$ should be a null-set with respect to 
{any}
non-atomic algebraic measure on $\bT^d$.
An algebraic measure is by definition a translate of a Haar measure on a closed subgroup of $\bT^d$. 
David Simmons found a gap in the proof presented in \cite{Sh1} and since then the third named author has been trying to determine whether the claim is true or not.

We mention that there is a subtle distinction between the full Haar measure $\lambda$ on $\bT^m$ and algebraic measures of smaller dimension. For the full Haar measure, the argument in \cite{Sh1} goes through with minor modifications to show that: if $\theta$ is non-singular, then $\lambda(\on{Bad}_\theta)=0$. For complete proofs, see \cite{Ki, Mo, MRS}.
Further, the recent paper \cite{Ki} also shows that for many\footnote{It is not clear at the moment if $\lam(\on{Bad}_\theta) = 1$ characterizes singularity
(although it seems unlikely to be the case). It is easy to see that $\lam(\on{Bad}_\theta)\in\seta{0,1}$.} singular $\theta$, $\lam(\on{Bad}_\theta) =1$.
See also the paper \cite{LSS} where related results on $\on{Bad}_\theta$ are obtained.

As is clear from the technique 
used in \cite{Sh1}, the richer the dynamics of the orbit of $x_\theta$ the better the chances to verify the aforementioned claim. Indeed, in an attempt to close the aforementioned gap, recent work of the second and third authors with N.\ Moshchevitin shows that once $\theta$ is not $k$-divergent for small $k$, $\on{Bad}_\theta$ is null with respect to non-atomic algebraic measures.

\begin{thm}[\cite{MRS}]\label{thm:MRS}
Let $\theta\in\bR^m$ be a vector which is not $k$-divergent for any {$k\le m-1$}. 
Then for any 1-dimensional algebraic measure $\mu$ on $\bT^m$,  $\mu(\on{Bad}_\theta) = 0$. 

More explicitly, given any $\eta \in \R^m$ and $\vb \in \Z^m$, we have for Lebesgue almost every $t \in \R$,
\begin{equation*} 
\liminf_{q\to\infty} \ q^{1/m} \idist{q\theta  - (\eta + t\tb{b})} = 0.
\end{equation*}
\end{thm}
To complement the story, we mention that in \cite{MRS} the authors also give a counterexample to the aforementioned claim from
\cite{Sh1} by constructing a non-singular vector $\theta \in \bR^3$ along with a $\eta\in \bR^3$ such that for {any} $t\in \bR$,
one has
$$ \liminf_{q\to\infty} \ q^{1/3} \idist{q\theta  -( \eta + t\ve_3)} \ge c$$
for some fixed $c>0$. Here $\ve_3 \in \R^3$ denotes the third vector in the standard coordinates. Thus, the algebraic measure corresponding to the translation by $\eta$ of the 1-dimensional subtorus of the third coordinate is {fully} supported inside $\on{Bad}_\theta$. By Theorem~\ref{thm:MRS} we must have that $\theta$ is either
$1$-divergent or $2$-divergent. For more information we refer the reader to \cite{MRS}.

\section{Invariance properties of $k$-divergent lattices}
In this section we prove certain invariance properties of $k$-divergent lattices. 
Let $\mathfrak{g}$ denote the Lie algebra of $G$ and consider the diagonal element $\va \in \mathfrak{g}$ for which $\exp(\va) = a_1 \in G$.
Consider the derivation $\text{ad}_\va:\mathfrak{g} \to \mathfrak{g}$. Clearly, this map is diagonalizable. 
We define
\begin{equation}
    \mathfrak{h}^+ := \text{span}_\R\left\lbrace \vb \in \mathfrak{g} : \text{ad}_\va(\vb) = \lambda \vb \text{ with } \lambda >0\right\rbrace,
\end{equation}
\begin{equation}
    \mathfrak{h}^- := \text{span}_\R\left\lbrace \vb \in \mathfrak{g} : \text{ad}_\va(\vb) = \lambda \vb \text{ with } \lambda <0\right\rbrace,
\end{equation}
\begin{equation}
    \mathfrak{h}^0 := \left\lbrace \vb \in \mathfrak{g} : \text{ad}_\va(\vb) = 0\right\rbrace.
\end{equation}
We thus have a vector space direct sum 
\begin{equation*}
    \mathfrak{g} = \mathfrak{h}^+ \oplus \mathfrak{h}^- \oplus \mathfrak{h}^0.
\end{equation*}
Moreover, each of the above summands is a Lie subalgebra of $\mathfrak{g}$.
Let $H^+, H^-, H^0$ be the corresponding connected Lie subgroups of $G$.
\begin{prop}\label{weak-stable-invariance}
Let $k \in \Z_{\geq 0}$. Let $x \in \D(k)$ and let $g^-$ and $g^0$ belong to $H^-$ and $H^0$ respectively.
We then have that
\begin{equation}
    \left(x \in \D(k)\right) \iff \left( g^- g^0 x \in \D(k) \right).
\end{equation}
\end{prop}
\begin{proof}
We work with elements $g^- = \exp(\vb)$ and $g^0= \exp(\vc)$ along with the assumptions $\ad_\va \vb = \lambda \vb$ for $\lambda <0$ and $\ad_\va \vc =0$. 
The reader is left to deduce the general case from this.
We then have, for $t>0$, 
\begin{equation}
\begin{split}
    a_t\exp(\vb)\exp(\vc)x &=  \left(a_t\exp(\vb) a_{-t}\right) \left( a_t \exp(\vc) a_{-t}\right) a_t x\\
    &= \exp\left(\Ad(\exp(t\va))\vb\right) \exp\left(\Ad(\exp(t\va))\vc\right) a_t x\\
    &= \exp\left(e^{t\ad_\va}\vb\right) \exp\left( e^{t\ad_{\va}}\vc\right) a_t x \\
    &= \exp\left(e^{t\lambda}\vb\right) \exp\left(\vc\right) a_t x.
    \end{split}
\end{equation}
The key observation is that $\lim_{t \to \infty}e^{t\lambda} = 0$.
Thus, for a divergent positive sequence $({t_n})$,
\begin{equation}\label{commutator}
    \left(\lim_{n\to\infty}a_{t_n} g^- g^0 x = y\right) \iff \left(\lim_{n\to \infty}a_{t_n} x = \left(\exp(-\vc\right)y\right).
\end{equation}
This immediately proves the proposition for the $\D(0)$ case.
For the $\D(k)$ case, we proceed by induction, assuming that the proposition has been proved for all $\D(j)$ with $j<k$.
Then, we can use the equivalence \eqref{commutator} along with the induction hypothesis that
\begin{equation*}
    \exp(\vc) \D(j) = \D(j) \text{ for } j<k.
\end{equation*}
This completes the proof and shows the so-called \textit{weak stable}
invariance property of $\D(k)$.
\end{proof}
Using the fact that $a_t$ has the form of \eqref{at-standard}, one can check that the groups defined above are actually given by
\begin{equation*}
    H^- = \left\lbrace  \left[ {\begin{array}{cc}
   I_m & 0  \\
   A & I_n 
  \end{array} } \right] : A \in \text{Mat}_{n\times m}(\R)\right\rbrace, 
\end{equation*}
\begin{equation*}
    H^0 = \left\lbrace  \left[ {\begin{array}{cc}
   A & 0  \\
   0 & B 
  \end{array} } \right] : A \in \text{Mat}_{m\times m}(\R), B \in \text{Mat}_{n \times n}(\R)\right\rbrace
\end{equation*}
and
\begin{equation*}
    H^+ = \left\lbrace  \left[ {\begin{array}{cc}
   I_m & A  \\
   0 & I_n 
  \end{array} } \right] : A \in \text{Mat}_{m\times n}(\R)\right\rbrace.
\end{equation*}
For future use, we use the notation
\begin{equation}\label{u_A}
    u_A := \left[ {\begin{array}{cc}
   I_m & A  \\
   0 & I_n 
  \end{array} } \right] \text{ and } x_A := u_A \Z^d.
\end{equation}

Our second invariance property arises from looking at the log-minima functions of a lattice.
We consider the supremum norm $\|\cdot \|$ on $\R^d$. Let $B(r)$ denote the supremum norm ball with radius $r$ and center at the origin. If $x$ is a lattice,
the $d$ successive minima functions of $x$ with respect to the supremum norm are defined as
\begin{equation*}
    \lambda_i(x) := \inf \left\lbrace r \in \R :  x\cap B(r) \text{ contains $i$ independent vectors}\right\rbrace.
\end{equation*}
The functions $\lambda_i : X_d \to \R$ are in fact continuous. We consider, for a fixed lattice $x \in X_d$ and the flow $(a_t)_{t\geq 0}$, the log-minima functions
\begin{equation}\label{log-minima-x}
    f_x(t): \R_{\geq 0} \to \R^d;\ t\mapsto \left(\log(\lambda_1(a_tx)), \dots, \log(\lambda_d(a_tx))\right).
\end{equation}
We denote the components of $f_x$ by $f_1, \dots, f_d$.
We also need to use the supremum norm to compare such functions on subsets of $\R$. We use the convention that, for $E \subset \R_{\geq 0}$ and $f,g: \R_{\geq 0} \to \R^d$,
\begin{equation}
    \|f-g\|_E := \sup_{t \in E} \|f(t) - g(t)\| = \sup_{t\in E} \left(\max_{i=1,\dots, d} |f_i(t)-g_i(t)|\right).
\end{equation}
For repeated implicit use in the future, we note the following reformulation of Mahler's compactness criterion:
\begin{lem}\label{Mahlerc}
    Let $x \in X_d$ be a lattice with log-minima function $f_x=(f_i)_{i=1,\dots,d}$. Then, for any sequence of positive times $(t_i)_{i \in \N}$, the sequence $(a_{t_i}x)_{i \in \N} \subset X_d$ has a convergent subsequence if and only if
    \begin{equation*}
        \limsup_{i \to \infty} f_1(t_i) > -\infty.
    \end{equation*}
\end{lem}
\begin{proof}
    We simply note that, for any $M \in \R$, Mahler's compactness criterion gives that
    \begin{equation}\label{mahlerset}
        \left\lbrace y \in X_d: \log(\lambda_1(y)) \geq M\right\rbrace
    \end{equation}
    is a compact set.
    Moreover, every compact subset of $X_d$ is contained in a set of the form \eqref{mahlerset}.
\end{proof} 
\begin{defn}
    We define an equivalence relation on the set of continuous functions $C(\R_{\geq 0}, \R^d)$ by writing $f \sim g$ if \begin{equation}\label{equivalence}
    \|f-g\|_{\R_{\geq 0}} < \infty.
\end{equation}
\end{defn}
\begin{prop}\label{k-div-invariance-fp}
Let $k \in \Z_{\geq 0}$, let $x \in \D(k)$ and let $y$ be another lattice in $X_d$. 
Further, let $f_x$ and $f_y$ denote the log-minima functions for $x$ and $y$ respectively. 
Then,
\begin{equation}
    \left(f_x \sim f_y\right) \implies \left(y \in \D(k)\right).
\end{equation}
\end{prop}
\begin{proof}
The proposition consists of $\N$-many assertions and we use induction (on $k$) to prove them all.

Assume $k=0$. By the definition of $\D(0)$ and Lemma \ref{Mahlerc},
\begin{equation*}
    \lim_{t\to\infty} \log(\lambda_1(a_t x))  = -\infty.
\end{equation*}
Since $f_x \sim f_y$, we see that $\lim_{t\to \infty}\log(\lambda_1(a_t y)) = -\infty$ as well.
Again, an application of Lemma \ref{Mahlerc} shows that $y\in \D(0)$.

Now assume $k>0$ and suppose our claim holds for every $j<k$. Let $C$ be a constant such that 
\begin{equation}\label{C}
    \|f_x-f_y\|_{\R_{\geq 0}} < C.
\end{equation}
Given that $x\in \D(k)$, we are to show that $y \in \D(k)$. Suppose $(t_i)_{i \in \N}$ is an unbounded positive sequence such that
\begin{equation}
    \lim_{i \to \infty} a_{t_i}y = w
\end{equation}
for some $w \in X_d$.
Using the joint continuity of the action of $(a_t)_{t \geq 0}$ on $X_d$, we may assume (perhaps after passing to a subsequence) that
\begin{equation}\label{atny-nclose-w}
    \|f_{a_{t_i}y} - f_w\|_{[0,i]} \leq 1 \ \text{ for every } i \in \N.
\end{equation}
Since $\{a_{t_i} y \in X_d : i \in \N \}$ is precompact and since $f_x \sim f_y$, Lemma \ref{Mahlerc} shows that a subsequence of $(a_{t_i}x)_{i \in \N}$ converges to some $z \in \D(j)$ (for $j <k$).
Again, passing to a subsequence, we can assume in addition to \eqref{atny-nclose-w} that
\begin{equation}\label{atnx-nclose-z}
    \|f_{a_{t_i}x} - f_z\|_{[0,i]} \leq 1.
\end{equation}
Applying \eqref{C}, \eqref{atny-nclose-w} and \eqref{atnx-nclose-z}, we get
\begin{equation}
\begin{split}
    \|f_w - f_z\|_{[0,i]} &= \|f_w - f_{a_{t_i}y} + f_{a_{t_i}y} - f_{a_{t_i}x} + f_{a_{t_i}x} - f_z\|_{[0,i]}\\
    &\leq 1 + C + 1
    \end{split}
\end{equation}
for every $i \in \N$.
Thus $\|f_z - f_w\|_{\R_{\geq 0}} < \infty$ and the induction hypothesis gives us that $w \in \D(j)$. Thus,
\begin{equation*}
    D(y) \subset \bigcup_{j<k} \D(j).
\end{equation*}
The argument to show that $D(y) \cap \D(k-1) \neq \varnothing$ is similar.
We are then finished in light of the characterization of $\D(k)$ as in \eqref{eq: D(k)-char}.
\end{proof}

\section{Templates}
Our proof of the non-emptiness and, furthermore, of a lower bound on the Hausdorff dimension of the set of $k$-divergent lattices is based on the results appearing in \cite{DFSU}. In order to present and then use their results, we first need to define the notion of templates.

\begin{defn}[Templates]\label{Templates}
For any real interval $I$, we use $I_\Z$ to denote its intersection with $\Z$. 
An $m \times n$ template is a continuous piecewise linear function\footnote{$f$ is piecewise linear means its domain is divided into a locally finite collection of closed intervals with disjoint interiors such that $f$ is linear on each one.} $f: [0,\infty) \xrightarrow[]{} \R^d$ (where $d=m+n$) that satisfies the following properties:
\begin{enumerate}
    \item[(a)] $f_1 \leq \dots \leq f_d$.
    \item[(b)] For all $i = 1,\dots,d$, we have $-\frac{1}{n} \leq f'_i \leq \frac{1}{m}$.
    \item[(c)] For all $j=0,\dots,d$ and for every interval $I$ such that $f_j<f_{j+1}$ we have that the function 
    \begin{equation*}\sum_{0<i \leq j} f_i
    \end{equation*}
    is convex and piecewise linear on $I$ with slopes in the set
    \begin{equation*}
        Z(j) := \left\lbrace \frac{L_+}{m} - \frac{L_-}{n}: L_+ \in [0,m]_{\Z},\ L_- \in [0,n]_\Z,\  L_+ + L_- = j \right\rbrace.
    \end{equation*} 
\end{enumerate}
As a convention we set $f_0=-\infty,f_{d+1}=\infty$.
\end{defn}
Templates are meant to model the log-minima functions of lattices. See \cite[Lemma 32.7]{DFSU} where a standard exterior algebra argument explains the relevance of condition (c) above.
We have the following theorem of fundamental importance:
\begin{thm}[Theorem 4.2 \cite{DFSU}]\label{template-lattice}
For $A \in \text{Mat}_{m\times n}(\R)$ let $f_A$ denote the log-minima function defined for the lattice $x_A$ defined in \eqref{u_A}.
Then there exists a template $f$ such that $f \sim f_A$.
Conversely, given a template $f$, there exists an $A \in \text{Mat}_{m\times n}(\R)$ for which $f_A \sim f$.
\end{thm}
\ignore{\begin{rem}
    Given a lattice $x$, we say that a $j$-dimensional subspace $V<\bR^d$ is $x$-rational if it is spanned by vectors from $x$. To an $x$-rational $j$-dimensional subspace $V$ there corresponds  a well defined vector $s_V\in \wedge^j\bR^d$. We call $V$ a shortest subspace
    if $s_V = \min\seta{||s_W||:W \text{ is $x$-rational $j$-dimensional}}$. It is convenient
    to use the $\sup$-norm on $\wedge^j\bR^d$ with respect to the standard basis since then the slopes of the function
    $||a_t s_V||$ are taken from the set $Z(j)$ from \eqref{eq: Z(j)}. 
     The standard way to get a template is to start from a lattice $x$ and to define the 
    $f_j$'s recursively by setting for $1\le j\le d$, $\sum_1^j f_j(x) :=  \|s_V\|$ where $V$ is a shortest $j$-dimensional $x$-rational. 
\end{rem}}
\begin{defn}[Score of a template]\label{score}
Let $f$ be a template and let $I$ be an interval on which it is linear. An \textit{interval of equality} for $f$ on $I$ is a set of integers $(p,q]_\Z$ where $0\leq p < q \leq d$ and where the components of $f$ satisfy
\begin{equation*}
    f_{p}< f_{p+1} = \dots = f_{q} < f_{q+1} \text{ on } I.
\end{equation*}
We define $M_{\pm}(p,q,I)$ to be the unique real numbers satisfying
\begin{equation}
    M_+(p,q,I) + M_-(p,q,I) = q-p\ \text{ and }\ \frac{M_+(p,q,I)}{m} - \frac{M_-(p,q,I)}{n} = \sum_{i=p+1}^q f'_i.
\end{equation}
It follows from part (c) of Definition \ref{Templates} that $M_\pm$ are positive integers. (See also the footnote on page 24 of \cite{DFSU}.)
Further define 
\begin{equation}
    S_+(f,I) := \bigcup \left\lbrace \left(p, p+M_+(p,q,I)\right]_\Z \subset (p,q]_\Z : (p,q]_\Z \text{ is an interval of equality}  \right\rbrace
\end{equation}
and
\begin{equation}
    S_-(f,I):= [1,d]_\Z \smallsetminus S_+(f,I).
\end{equation}
We also have 
\begin{equation}
    \delta(f,I):= \#\left\lbrace (i_+, i_-) \in S_+(f,I)\times S_-(f,I): i_+< i_- \right\rbrace.
\end{equation}
We take $\delta(f,t)$ to be the piecewise constant function with value $\delta(f,I)$ on $I$.
Finally,
\begin{equation}
    \underline{\delta}(f) := \liminf_{T \to \infty} \frac{1}{T} \int_0^T \delta(f,t)dt.
\end{equation}
\end{defn}
We aim to compute the Hausdorff dimension of the set of matrices $A \in \text{Mat}_{m \times n}(\R)$ for which $x_A$ is $k$-divergent. 
Given a template $f$, we define the set
\begin{equation}
    \mathcal{D}(f) := \left\lbrace A \in \text{Mat}_{m\times n}(\R) : f_A \sim f \right\rbrace.
\end{equation}
We say that a collection of templates $\mathcal{F}$ is \textit{closed under finite perturbation} if 
\begin{equation}
    \left(f \in \mathcal{F} \text{ and } f' \text{ is a template with } f'\sim f \right) \implies \left(f' \in \mathcal{F}\right).
\end{equation}
For such a $\mathcal{F}$, we define
\begin{equation}
    \mathcal{D}(\mathcal{F}) := \bigcup_{f \in \mathcal{F}} \mathcal{D}(f) \subset \operatorname{Mat}_{m \times n}(\R).
\end{equation}
A tool to compute dimension is:
\begin{thm}[Theorem 4.3 \cite{DFSU}]\label{DFSU-dimension-thm}
Let $\mathcal{F}$ be a Borel\footnote{Under the compact-open topology.} collection of templates closed under finite perturbation.
Then we have that the Hausdorff dimension of $\mathcal{D}(\mathcal{F})$ is
\begin{equation}
    \sup\left\lbrace \underline{\delta}(f): f \in \mathcal{F}\right\rbrace.
\end{equation}
\end{thm}

\section{Existence of $k$-divergent lattices}

Our goal is to build, for each $k \in \N$, an example of a template $f^{(k)}$ giving rise to $k$-divergent lattices via Theorem \ref{template-lattice}. 
Following this, we apply Theorem \ref{DFSU-dimension-thm} to the equivalence class of templates
\begin{equation}
    \mathcal{F}^k := \left\lbrace f \in C(\R_{\geq 0}, \R^d) : f \text{ is a template with } f\sim f^{(k)} \right\rbrace.
\end{equation}

We have the following definition which is to be the input for the inductive procedure giving rise to the templates $f^{(k)}$.
\begin{defn}[Linked Templates]
    A linked template is a pair $(f, \mathcal{I})$ where $f$ is a template and $$\mathcal{I} = \{I_p:=[b_p,b_{p+1}]: p \in \N\}$$
     is a collection of intervals satisfying
     \begin{equation*}
         b_1 =0,\  b_p \nearrow \infty\ \text{ and }\ f(b_p) = 0\in \R^d \text{ for all } p \in \N.
     \end{equation*}
\end{defn}

Our first example of a linked template is built with the following block:
\begin{defn}\label{standard-template}
We define a piecewise linear function $g:[0,m+n] \to \R^d$ by specifying its coordinates.
We have
\begin{equation}
    g_1(t) = \begin{cases} -\frac{t}{n} & \text{ if }\ 0\leq t \leq n
     \\
     \frac{t-(m+n)}{m} & \text{ if }\ n \leq t \leq m+n
     \end{cases}
\end{equation}
and
\begin{equation}
    g_2(t) = \begin{cases} \frac{t}{n(m+n-1)} & \text{ if }\ 0\leq t \leq n
     \\
     \frac{m+n-t}{m(m+n-1)} & \text{ if }\ n \leq t \leq m+n.
     \end{cases}
\end{equation}
We set $g_3 = \dots = g_{m+n}=g_2$ (if $d>2$).
We leave to the reader the straightforward check that $g:= (g_i)_{i=1,\dots, d}$ satisfies the template axioms  of Definition \ref{Templates} on its domain.
\end{defn}
This example of a linked template gives rise to a template associated to $1$-divergent lattices.
\begin{exmp}[$1$-divergent lattice template]\label{1div-template}
We define a family of closed intervals $$\mathcal{I} =  \{I_p \subset [0,\infty): p \in \N \}$$ inductively. Let $I_1 = [0,d]$ and let $I_p$ have length $pd$ and be contiguous to $I_{p-1}$.

For each $p \in \N$, let $\psi_p$ be the orientation preserving linear bijection $I_p \to [0,d]$ and let $\phi_p: \R^d\to \R^d$ be the linear map which scales vectors by the factor $p$.
Define
\begin{equation}\label{1div-blocks}
    f_{p} := \phi_p \circ g \circ \psi_p :I_p \to \R
\end{equation}
where $g$ is as in Definition \ref{standard-template}. Since $f_{p}$ takes the value $0 \in \R^d$ on both endpoints of $I_p$, we define $f:[0,\infty) \to \R^d$ to be the unique continuous function which  restricts to $f_{p}$ on each $I_p$.
Since the derivatives of $\psi_p$ and $\phi_p$ are $p^{-1}$ and $pI$ respectively, we see that $f$ is a template.
Note that with respect to part (c) of Definition \ref{Templates}, the components of $f$ satisfy inequalities only on intervals avoiding the endpoints of the $I_p$.

Moreover, by construction, $(f, \mathcal{I})$ is a linked template.
In the case when $m=2$ and $n=1$ we have the following graph of $f$.
\begin{center}\label{graph}
    \includegraphics[scale=0.6]{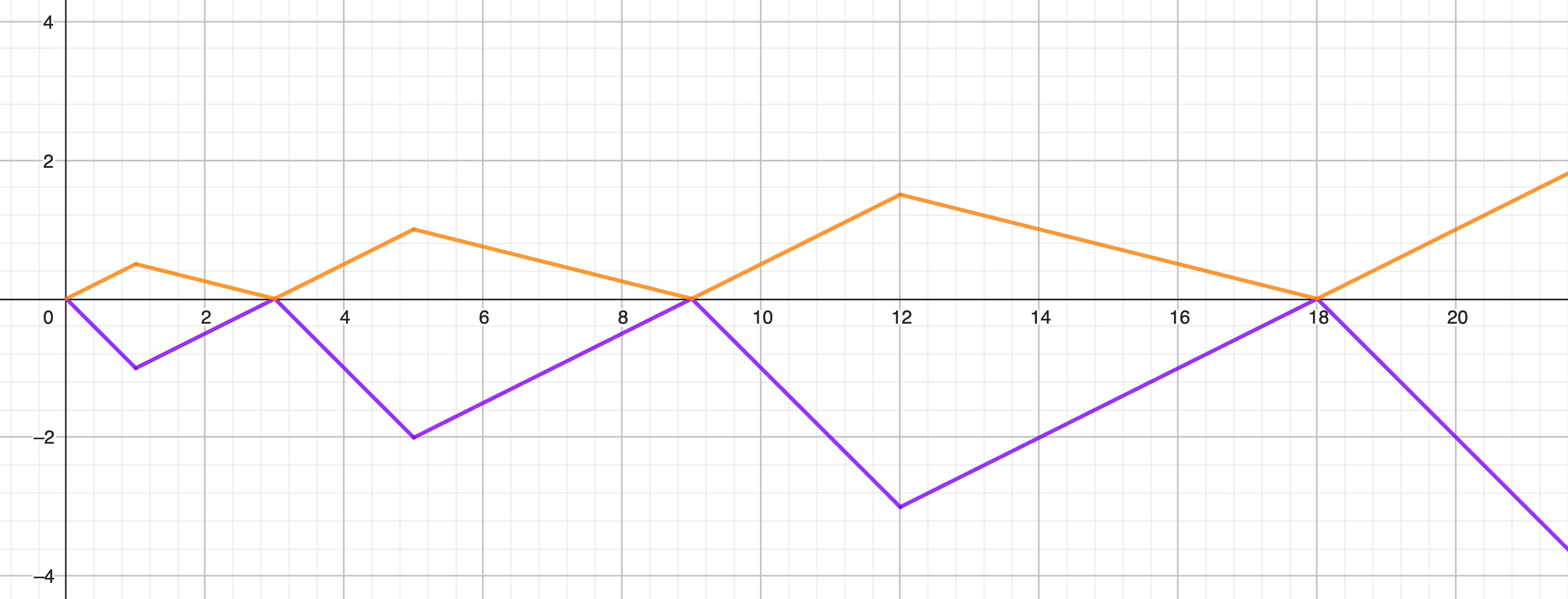}
\end{center}
 The purple graph is of $f_1$ and the orange graph is of $f_2 = f_3$.
\end{exmp}

\begin{thm}\label{1template-is-actually-1template}
Fix the template $f:[0,\infty) \to \R^d$ as in Example \ref{1div-template}. Let $x \in X$ be a lattice with log-minima function $f_x$.
Assume further that $f_x \sim f$. Then $x \in \D(1)$.
\end{thm}
\begin{proof}
Fix a constant $C$ such that $\|f_x -f\|_{\R_{\geq 0}} \leq C$.
This and Lemma \ref{Mahlerc} show that $D(x)$ is nonempty. In order to show $x \in \D(1)$, we show that any $y \in D(x)$ has a divergent orbit. So let $y \in D(x)$ and let $f_y$ be the log-minima function of $y$.
Again, using the continuity of the action, we can find a sequence of times $(t_i)_{i \in \N}$ going to infinity such that
\begin{equation}\label{atix-close-y}
    \|f_{a_{t_i}x} - f_y\|_{[0,i]} \leq 1.
\end{equation}
Here $f_{a_{t_i}x}$ denotes the log-minima function of the lattice $a_{t_i}x$.
Recall now that $d=m+n$ is the dimension of the ambient space of $x$.
\begin{claim}\label{1Div-claim}
Fix any $M>0$.
Given any $T_0>0$ we have, for all sufficiently large $t$,
\begin{equation}\label{x-leaves-K-for-long-after-finite-time}
    \left(\log(\lambda_1(a_tx)) > -M\right) \implies \left(\sup \left\lbrace \log(\lambda_1(a_{t+s} x)) : s \in  [Md, Md + T_0]\right\rbrace  \leq -M\right)
\end{equation}
\end{claim}
\begin{proof}
Let $f_1$ denote the first component of $f$. 
By the very definition of $f$, we see that the claim (with $M$ replaced by $M+C$) holds for $f_1$.
Since $\|f_x -f\|_{\R_{\geq 0}} < C$, we see that the claim holds for the function $t\mapsto \log(\lambda_1(a_tx))$ as well.
\end{proof}
We now show that 
\begin{equation*}
    \lim_{t\to \infty} \log(\lambda_1(a_ty)) = -\infty,
\end{equation*}
that is, $y$ has divergent orbit.
Let $M>0$ be any positive constant large enough so that $\log(\lambda_1(y)) > -(M-1)$ and let $T_0>0$ be arbitrary. 
By \eqref{atix-close-y} we see that, for every $i \in \N$,
\begin{equation*}
    \log(\lambda_1(a_{t_i} x)) > -M.
\end{equation*}
We apply claim \ref{1Div-claim} and see that, for all sufficiently large $i \in \N$, 
\begin{equation*}
    \log(\lambda_1(a_{t_i+s}) x)  \leq -M \text{ for every } s \in [Md,Md+T_0].
\end{equation*}
Another application of \eqref{atix-close-y} gives that
\begin{equation*}
\log(\lambda_1(a_sy)) \leq -(M-1) \text{ for all } s \in [Md, Md+T_0].
\end{equation*}
Since $T_0$ was arbitrary, this proves the required divergence property of the function $t\mapsto \log(\lambda_1(a_ty))$.
\end{proof}
\ignore{The equivalence $f_x \sim f$ and Mahler's compactness criterion show that $D(x)$ is nonempty. Let $y$ be a lattice in $D(x)$.
Let $f_y$ be the log-minima function of $y$. Denote the first coordinate of $f_y$ by ${f_y}^{(1)}$ (i.e the log of $\lambda_1(a_{t}y)$). We show that ${f_y}^{(1)}(t)$ goes to $-\infty$ when $t \xrightarrow[]{} \infty$. Assume on the contrary that there exists some compact set $K$ in the space of lattices, and an infinite sequence $t_n$, such that $a_{t_n}y \in K$ We can also take $K$ to be bigger and assume w.l.o.g that $y \in K$ (and even in the interior of $K$), and that $K= \brace z: \lambda_{1}(z)> \delta \brace$ (Using Mahler's criterion).  Now since $y \in D(x)$ there exists a sequence $t_s$ of times such that $a_{t_s}x \xrightarrow[]{} y$. Now, for every $n$ big enoug, we can find some $s$ big enough, such that $a_{t_s}x$ is so close to $y$, so that $a_{t_n}a_{t_s}x$ and $a_{t_n}y$ remain close. Now, since $y$ lies in the interior of $K$ and $a_{t_s}x$ is close to it, we have $a_{t_s}x \in K$. If we translate this to the language of first minima functions, this is equivalent to saying ${f_{x}}^{(1)}(t_s)<log(\delta)$. Now since $f_x \sim f$, there exists some constant $C>0$ such that $|f_x - f|<C$. Hence $f^{(1)}(t_s)<log(\delta)+C$. From the construction of $f$, we know that given any pair $r \leq R$, there exists some $L$, such that for every $l>L$ there exists $T_l$ such that for every $t>T_l$ if $f^{1}(t)<r$, then $f^{1}(t+L)>R$. We now apply this fact to $f$ with $r=log(\delta)+C,R=r=log(\delta)+2C$ and take $L=t_n$ (choosing $n$ to be big enough w.r.t $K$) and $t_s>T_{t_t}$ (choosing $s$ to be big enough w.r.t $n$). Now, we obtain that since $f^{(1)}(t_s)<log(\delta)+C$ then $f^{(1)}(t_s+t_n)>log(\delta)+2C$, and therefore $f_{x_{t_s}}^{(1)}(t_n)=f_{x}^{(1)}(t_s+t_n)>log(\delta)+C$ Hence $f_y ^{(1)}(t_n)>log(\delta)$, contradicting our assumption that $a_{t_n}y \in K=\brace z:\lambda_1(z)<\delta \brace$. 
Hence our template $f$ is 1-divergent and we are done.}
We define an operator $\Phi$ from the set of linked templates to itself, the iterates of which will help us in our inductive construction. The following definition should be read as 
a guided exercise and the reader should fill in the details.
\begin{defn}\label{marked-operator}
    We set
\begin{equation*}
    \Phi(f,\mathcal{I}) := (\widetilde{f}, \widetilde{\mathcal{I}})
\end{equation*}
where $\widetilde{f}$ and $\widetilde{\mathcal{I}}$ are defined as follows:
Let $\widetilde{I}_1:= I_1$ and, assuming we have defined 
\begin{equation*}
    \widetilde{I}_1 = [c_1,c_2], \widetilde{I}_2 = [c_2,c_3],\dots, \widetilde{I}_{q} = [c_q,c_{q+1}],
\end{equation*}
we set
\begin{equation*}
    \widetilde{I}_{q+1} := c_{q+1} + \left(I_1 \cup I_2 \cup \dots \cup I_q \cup I_{q+1}\right).
\end{equation*}
Here is a representation:
\begin{center}
    \includegraphics[scale=0.2]{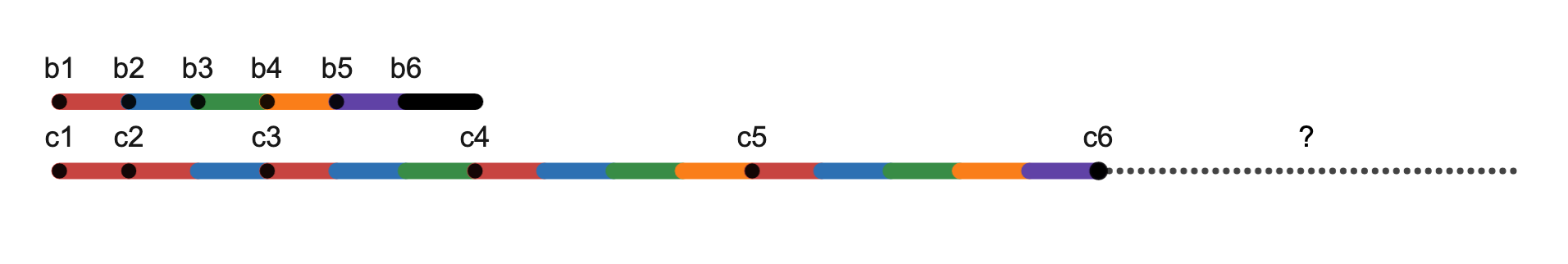}
\end{center}
Can you guess the colors that go in the empty bin?

We now define $\widetilde{f}: \R_{\geq 0} \to \R^d$ to be the unique function which satisfies the property that, for every $q \in \N$ and $t \in \widetilde{I}_{q}$,
\begin{equation}\label{fk-formula}
    \widetilde{f}(t) = f(t-c_{q}).
\end{equation}
This function $\widetilde{f}$ is continuous by the hypothesis that $f$ takes value $0 \in \R^d$ on the endpoints of $I_p$.
The template axioms can be easily checked by using the fact that $f$ itself is a template. Note that with respect to part $(c)$ of Definition \ref{Templates}, if for some $j$, $\widetilde{f}_j < \widetilde{f}_{j+1}$ on an interval $I$, the linked template definition implies that $I \subset \widetilde{I}_q$ for some $q \in \N$.
Moreover, the same property holds for $f$.

The pair $(\widetilde{f}, \widetilde{I})$ is another linked template.
\end{defn}
We now construct putative examples of templates associated to $k$-divergent lattices.
\begin{exmp}[$k$-divergent lattice template]\label{k-div-template}
Let $(f^{(1)},\mathcal{I}^{(1)})$ be our linked template from example \ref{1div-template}.
 Define
 \begin{equation*}
     (f^{(k)},\mathcal{I}^{(k)}) := \Phi^k (f^{(1)}, \mathcal{I}^{(1)}).
 \end{equation*}
\end{exmp}
\begin{thm}\label{existence-of-kdiv}
Fix $k \in \N$. 
Let $x \in X_d$ be a lattice with log-minima function $f_x$. Assume further the equivalence $f_x \sim f^{(k)}$ where $f^{(k)}$ is constructed as in Example \ref{k-div-template}.
Then $x \in \D(k)$.
\end{thm}
\begin{proof}
We proceed by induction on $k$ with the base step having already been proved in Theorem \ref{1template-is-actually-1template}. 
Thus we can assume $k>1$. Let $C$ be a constant witnessing $f_x \sim f^{(k)}$.
Let $y$ belong to $D(x)$. 
As before, we can find a sequence of positive unbounded times $(t_i)_{i \in \N}$ with 
\begin{equation}\label{k-div-y-iclose}
    \|f_{a_{t_i}x} - f_y\|_{[0,i]} < 1.
\end{equation}
Again, $f_{a_{t_i}x}$ denotes the log-minima function of the lattice $a_{t_i}x$.

Denote $(f^{(k-1)}, \mathcal{I}^{(k-1)})$ by $$(f, \{I_p:=[b_p,b_{p+1}]: p \in \N\})$$
and $(f^{(k)}, \mathcal{I}^{(k)})$  by 
\begin{equation*}
    (\widetilde{f}, \{\widetilde{I}_q:=[c_q,c_{q+1}]: q \in \N\}).
\end{equation*}
For each $i \in \N$, let $q_i$ be the unique integer such that $t_i \in [c_{q_i}, c_{q_i +1})$. Assume for now that
\begin{equation}\label{in-a-single-J}
    [t_i, t_i +i] \subset [c_{q_i},c_{q_i+1}].
\end{equation} 
Consider the differences $(t_i - c_{q_i})$ and the two possibilities:
\begin{enumerate}
    \item[(a)] The sequence $(t_i - c_{q_i})_{i \in \N}$ has a bounded subsequence. Without loss of generality, assume the sequence itself is bounded. We compute that for each $i\in\N$ and $t \in [0,i]$,
    \begin{equation}\label{y-k-1-div-computation}
        \begin{split}
            \|f_y(t) - f(t)\| &\leq \|f_y(t) - f_{a_{t_i}x}(t)\| + \| f_x(t_i +t) - f(t)\|\\
            &= 1 + \| f_x(t_i +t) - \widetilde{f}(t_i +t)\| + \|\widetilde{f}(t_i+t) -f(t)\|\\
            &\leq 1 + C + \|\widetilde{f}(t_i +t) - f(t)\|\\
            &= 1+C + \|f(t_i - c_{q_i} + t) - f(t)\|,
        \end{split}
    \end{equation}
    where the last equality follows from using the assumption of \eqref{in-a-single-J}
    and the defining property of $\Phi$ as in \eqref{fk-formula}.
    Since $f$ is continuous and piecewise linear with derivatives of its components bounded in absolute value by $\max\left\lbrace\frac{1}{m}, \frac{1}{n}\right\rbrace$, and since $(t_i - c_{q_i})_{i \in \N}$ is a bounded sequence, we see that the last term in \eqref{y-k-1-div-computation} has a uniform bound over all $i \in \N$.
    This shows that $f_y \sim f^{(k-1)}$ with the induction hypothesis kicking in to give
    \begin{equation*}
        y \in \D(k-1).
    \end{equation*}
    \item[(b)] The sequence $(t_i - c_{q_i})_{i \in \N}$ diverges to infinity. For each $i$ and $t \in [0,i]$, we get  the following chain of inequalities
    \begin{equation}\label{y-<k-1-div-computation}
    \begin{split}
     \|f_y(t) - f(t_i - c_{q_i} + t)\| 
         &\leq 1 + \|f_x(t_i + t) - f(t_i - c_{q_i} + t) \| \\
        &\leq 1 + C + \|\widetilde{f}(t_i +t) - f(t_i - c_{q_i} + t) \| \\
        & = 1 + C,
    \end{split}
    \end{equation}
    where the last equality follows from the assumption in \eqref{in-a-single-J} and the defining property in \eqref{fk-formula}.
    Since $f$ is a template, Theorem \ref{template-lattice} gives the existence of a lattice $z$ for which $f_z \sim f (= f^{(k-1)})$.
    The induction hypothesis implies that $z \in \D(k-1)$. 
    The computation in \eqref{y-<k-1-div-computation} shows that 
    \begin{equation*}
        \|f(t_i-c_{q_i})\| \geq \|f_y(0)\| - (1+C) \text{ for all } i \in \N.
    \end{equation*}
    Thus, the equivalence $f_z\sim f$ and Lemma \ref{Mahlerc} shows that the sequence of lattices $a_{(t_i -d_{q_i-1})}z$ has a convergent subsequence. Let $w \in D(z)$ be the limit.
    The computation \eqref{y-<k-1-div-computation} shows that that $f_y \sim f_w$.
    Theorem \ref{k-div-invariance-fp} then shows that 
    \begin{equation}
        y \in  \bigcup_{j<k-1} \D(j),
    \end{equation}
    since $w$ lies in the same set.
\end{enumerate}
If \eqref{in-a-single-J} fails to hold, we can assume without loss of generality that, there exists $s \in \R$ such that
\begin{equation*}
    |t_i - c_{q_i +1}| < s \text{ for all } i \in \N.
\end{equation*}
But then we can replace $t_i$ by $t_i + s$ and $y$ by $a_s y$ to return to the situation of \eqref{in-a-single-J}. Note we use here that the divergence properties of $y$ are invariant under the flow.

Summing up, we have shown that 
\begin{equation}
    D(x) \subset \bigcup_{j<k} \D(j),
\end{equation}
and it remains to show that $D(x) \cap \D(k-1)$ is nonempty.

And this is straightforward: let $(t_i)_{i \in \N}$ be the sequence of initial points of the intervals $\widetilde{I}_q$, that is $(c_q)_{q \in \N}$.
Since $f^{(k)}$ takes value $0 \in \R^d$ at these points, Mahler's compactness shows that $(a_{t_i}x)_{i \in \N}$ has a convergent subsequence. Call the limit $y$.
By construction, the log-minima function $f_y$ is equivalent to $f^{(k-1)}$.
The induction hypothesis shows that $y \in \D(k-1)$ and we are done.
\end{proof}
\section{Computations for a lower bound on dimension}
We can now turn to computing the Hausdorff dimension of $\D(k)$.
Let $\mathcal{F}^k$ be the set of templates which are equivalent to $f^{(k)}$. It is clearly closed under finite perturbations and Borel in the compact-open topology on $C(\R_{\geq 0},\R^d)$.
Each $A \in \mathcal{D}(\mathcal{F}^k)$ gives rise to a lattice $x_A$ which is necessarily in $\D(k)$, by Theorem \ref{existence-of-kdiv}.
Thus, in light of Proposition \ref{weak-stable-invariance}, computing the dimension of $\mathcal{D}(\mathcal{F}^k) \subset \text{Mat}_{m\times n} (\R)$ will lead to a lower bound for the dimension of $\D(k) \subset X_d$.
\begin{lem}
The score $\underline{\delta}(f^{(k)})$ is equal to $\delta_{m,n}$ which is defined by
\begin{equation}\label{eq:deltamn}
    \delta_{m,n} := mn - \frac{mn}{m+n}.
\end{equation}
\end{lem}
\begin{proof}
We note that the first part of this computation has already been done in \cite[Figure 4]{DFSU}.
We proceed by induction and begin with the case $k=1$. We must contend with the function $\delta(f^{(1)},t)$ and thus also the quantities $S_+, M_+$ and $M_-$ from Definition \ref{score}.
Since $f^{(1)}$ is defined by repeating the same pattern over larger and larger scales, it suffices to compute $\delta(f^{(1)}, t)$ over the interval $[0,d]$.
First consider the interval $K_1 := [0,n]$. There are two intervals of equality here, $(0, 1]_\Z$ and $(1,d]_\Z$.
Solving the system of equations
\begin{equation*}
    M_+(0,1,K_1) + M_-(0,1,K_1) = 1\ \text{ and }\ \frac{M_+(0,1,K_1)}{m} - \frac{M_-(0,1,K_1)}{n} = -\frac{1}{n},
\end{equation*}
we see that $M_+(0,1,K_1) = 0$.
Solving
\begin{equation*}
    M_+(1,d,K_1) + M_-(1,d,K_1) = d-1\ \text{ and }\ \frac{M_+(1,d,K_1)}{m} - \frac{M_-(1,d,K_1)}{n} = \frac{1}{n},
\end{equation*}
we see that $M_+(1,d,K_1) = m$.
We thus have
\begin{equation*}
    S_+(f^{(1)}, K_1)  = [2,m+1]_\Z\ \text{ and }\ S_-(f^{(1)}, K_1) = 
    \{1\} \cup [m+2, m+n]_\Z.
\end{equation*}
This shows that
\begin{equation}\label{delta-0n}
    \delta(f^{(1)}, t) = m(n-1)\ \text{ for }\ t \in [0,n].
\end{equation}
We pause here to note that the interval $[m+2, m+n]_\Z$ could be empty; for example, in the case when $m=n=1$.

Now consider the interval $K_2:=[n,d]$. We again have two intervals of equality, $(0,1]_\Z$ and $(1,d]_\Z$.
One can perfom a similar computation as above to see that $M_+(0,1,K_2) = 1$ and $M_+(1,d,K_2) = m-1$.
We thus have
\begin{equation*}
    S_+(f^{(1)}, K_2)  = [1,m]_\Z\ \text{ and }\ S_-(f^{(1)}, K_2) = 
     [m+1, m+n]_\Z.
\end{equation*}
This shows that
\begin{equation}\label{delta-nd}
    \delta(f^{(1)}, t) = mn\ \text{ for }\ t \in [n,d].
\end{equation}
Recall that $f^{(1)}$ was defined over a specific collection of intervals which we now call
\begin{equation*}
    \left\lbrace I_p = [b_p, b_{p+1}]: p \in \N \right\rbrace. \end{equation*}
    Each $I_p$ has length $pd$.
    The definition in \eqref{1div-blocks} and equations \eqref{delta-0n} and \eqref{delta-nd} show that, for any $q \in \N$,
    \begin{equation}
        \int_{I_q} \delta(f^{(1)},t) dt = qn (mn - m) + qm (mn).
    \end{equation}
    Thus, for any $T \in I_p$ with $p >1$, we see that 
    \begin{equation}\label{score-calculation-1}
    \begin{split}
        \frac{1}{T}\int_0^T \delta(f^{(1)}, t) dt &=  \frac{1}{T} \left( \int_0^{b_p} \delta(f^{(1)},t) dt + \int^T_{b_p}\delta(f^{(1)},t) dt\right)
        \\ 
        &=  \frac{1}{T}  \int_0^{b_p} \delta(f^{(1)},t) dt + \frac{1}{T} \int^{T}_{b_p} \delta(f^{(1)}, t) dt
        \\
        &= \frac{1}{T} \sum_{q=1}^{p-1} \left(qmn^2 + qm^2n - qmn  \right) + \frac{1}{T} \int_{b_p}^T \delta(f^{(1)},t)dt
        \\
        &= \frac{mn^2 + m^2n - mn}{d}\cdot \frac{1}{T} \sum_{q=1}^{p-1} qd + \frac{1}{T} \int_{b_p}^T \delta(f^{(1)},t)dt
        \\
        &= \delta_{m,n} \cdot \frac{1}{T} \sum_{q=1}^{p-1} qd + \frac{1}{T} \int_{b_p}^T \delta(f^{(1)},t)dt.
        \end{split} 
    \end{equation}
Since
\begin{equation}\label{add-insignificant}
    \frac{\text{length}\left(\bigcup^{p-1}_{q=1}\ I_q\right)}{\text{length}\left(\bigcup^{p}_{q=1}\ I_q\right)} = \frac{p-1}{p+1}
\end{equation}
and since $\delta(f^{(1)},t)$ is a bounded function,
computation \eqref{score-calculation-1} shows that
\begin{equation}
    \lim_{T\to \infty} \frac{1}{T}\int_0^T \delta(f^{(1)}, t) dt = \delta_{m,n}.
\end{equation}

Now we come to $f^{(k)}$ where $k>1$. Assume the specified decomposition of its domain is denoted by
\begin{equation*}
    \{J_p =[c_p,c_{p+1}]: p \in \N\}.
\end{equation*}
First, by recalling the formula of \eqref{fk-formula}, one proves inductively on $k$ that
\begin{equation}
    \int_{J_p} \delta(f^{(k)},t) dt = \text{length}(J_p) \left(\frac{n}{m+n} (mn-m) + \frac{m}{m+n} (mn)\right).
\end{equation}
Second, we leave it to the reader as an exercise to prove the limit
\begin{equation}\label{add-insignificant-k}
    \lim_{p\to \infty}\ 
    \frac{\text{length}\left(\bigcup^{p-1}_{q=1}\ J_q\right)}{\text{length}\left(\bigcup^{p}_{q=1}\ J_q\right)}
    = 1
\end{equation}
using induction on $k$, with the base step being seen by \eqref{add-insignificant}.
Hint: the induction hypothesis guarantees, for any $i \in \N$, that
\begin{equation*}
    \lim_{p \to \infty} \frac{\operatorname{length}(J_{p})}{\operatorname{length}(J_{p+i})} = 1.
\end{equation*}
We can now take $T \in J_p = [c_p, c_{p+1}]$ and then compute, exactly as in \eqref{score-calculation-1},
\begin{equation}
    \frac{1}{T} \int_0^T \delta(f^{(k)},t) dt = \delta_{m,n} \cdot\frac{\text{length}\left(\bigcup_{q=1}^{p-1} J_q\right)}{T} + \frac{1}{T} \int_{c_p}^T \delta(f^{(k)},t) dt.
\end{equation}
Equation \eqref{add-insignificant-k} then shows that 
\begin{equation*}
    \lim_{T\to \infty} \frac{1}{T} \int_0^T \delta(f^{(k)}, t) dt = \delta_{m,n}.
\end{equation*}
This completes the proof.
\end{proof}
Theorem \ref{DFSU-dimension-thm} and the above score computation immediately yields:
\begin{cor}\label{cor-lower bound 1}
The dimension of $\mathcal{D}(\mathcal{F}^k) \subset \text{Mat}_{m\times n}(\R)$ is bounded below by $\delta_{m,n}$.
\end{cor}
\begin{cor}\label{lower-bound-cor}
Let $k \in \N$. The dimension of $\D(k) \subset X_d$ is bounded below by
\begin{equation}
    \dim(X_d) - \frac{mn}{m+n}.
\end{equation}
\end{cor}
\begin{proof}
By Proposition \ref{weak-stable-invariance} we have
\begin{equation*}
    \D(k) = H^- H^{0} \D(k).
\end{equation*}  
As a consequence, since locally $G$ is a metric product of small neighbourhoods of the identity
in the groups $H^-, H^0, H^+$, it follows that 
\begin{equation*}
    \dim \D(k) = \dim (\D(k) \cap H^+ \Z^d) + \dim (H^- H^0).
\end{equation*}
Corollary \ref{cor-lower bound 1} says that  $\delta_{m,n}$ is a lower bound for the dimension 
of $\D(k) \cap H^+ \Z^d$ and so we see that $\delta_{m,n} + \dim (H^{-} H^0)$ is a lower bound for $\D(k)$. Since 
\begin{equation*}
    \dim(X_d) = d^2 - 1\ \text{ and }\ \dim(H^-H^0) = d^2 - 1 - mn,
\end{equation*}
we are done.
\end{proof}

\section{Full escape of mass and an upper bound on dimension}

Let $\mathcal{M}(X_d)$ and $\mathcal{P}(X_d)$ denote the set of Borel complex measures on $X_d$ and the set of Borel probability measures on $X_d$ respectively. Every open set in $X_d$ is $\sigma$-compact and so, if we consider the $C^*$-algebra of continuous complex-valued functions vanishing at infinity, $C_0(X_d)$, the Riesz Representation theorem \cite[Theorem 6.19]{Ru} asserts that $\mathcal{M}(X_d)$ is in bijection with the continuous dual $C_0(X_d)^{\vee}$, via
\begin{equation}
    \nu \mapsto \left(f\mapsto \int f d\nu\right).
\end{equation}
We induce the weak$^*$-topology on $\mathcal{M}(X_d)$ and the corresponding subset topology on $\mathcal{P}(X_d)$.
Given $x\in X_d$, we define the family of probabilities $\{\mu_{x,T} \in \mathcal{P}(X_d): T>0\}$ by
\begin{equation}\label{mu_T}
    \int f d\mu_{x,T} := \frac{1}{T}\int_0^T f(a_tx) dt
\end{equation}
where $f \in C_0(X_d)$.
The following notion of divergence on average is crucial for us:
\begin{defn}[cf. \cite{KKLM} Theorem 1.1]
    We say $x \in X_d$ is divergent on average if
    \begin{equation*}
        \lim_{T \to \infty} \mu_{x,T} = 0\ \text{ in } \mathcal{M}(X_d).
    \end{equation*}
\end{defn}
\begin{thm}\label{escape}
Fix $k \in \N$. Let $x\in \D(k)$. Then, $x$ is divergent on average.
\end{thm}
\begin{proof}
Note that the Banach-Aloaglu theorem implies that every sequence $(\mu_{x,T_i})_{i \in \N}$ has a limit point in $\mathcal{M}(X_d)$.
Thus, to prove the theorem it suffices to show that for any sequence $(T_i)_{i \in \N}$ of divergent times, 
\begin{equation}
   \left( \lim_{i \to \infty}\mu_{x,T_i} = \mu\ \text{ in } \mathcal{M}(X_d) \right) \implies \left(\mu = 0\right).
\end{equation}
Note that the limit measure $\mu$ in the equation above is necessarily non-negative, and also invariant under the action of $(a_t)_{t \in \R}$.
Thus, if we have a limit $\mu$ which is nonzero, we may scale to obtain an $(a_t)_{t \in \R}$-invariant probability measure $\nu$.
Assume we have such a $\nu$ for sake of contradiction.

First observe that $\text{supp}(\nu) = \text{supp}(\mu)$ and so any $y \in \text{supp}(\nu)$  must belong to $D(x)$. In particular, there must be some $j<k$ for which
\begin{equation*}
    y \in \D(j).
\end{equation*}

On the other hand, we claim that we can find a $y \in \text{supp}(\nu)$ for which there exists a sequence of unbounded positive times $(t_i)_{i \in \N}$ with 
\begin{equation}\label{y-recurrent}
    \lim_{t \to \infty} a_{t_i}y = y.
\end{equation}
To see this, consider a countable base for the topology of $X_d$, $\{U_l \subset X_d: l \in \N\}$, and consider the sets
\begin{equation*}
    A_l := \left\lbrace z \in U_l : a_{j} z \notin U_l \text{ for all } j \in \N \right\rbrace.
\end{equation*}
Poincare recurrence shows that $\nu\left(\cup A_l\right) = 0$.

Thus, we can choose any $y \in \text{supp}(\nu) \smallsetminus \cup A_l$ and get the claim in \eqref{y-recurrent}.
This contradicts the fact that $y \in \D(j)$.
\end{proof}
Applying \cite[Theorem 1.1]{KKLM}, we immediately get:
\begin{cor}\label{upper-bound-cor}
For each $k \in \N$, the Hausdorff dimension of $\D(k)$ is bounded above by
\begin{equation*}
    \dim(X_d) - \frac{mn}{m+n}.
\end{equation*}
\end{cor}

\subsection*{Acknowledgements}
The second, third and fourth name authors acknowledge generous support from the European Research Council (ERC) under the European Union’s Horizon 2020 Research and Innovation Program, Grant agreement no.\ 754475.
The authors thank the organizers of the 2022 conference `Group Actions, Geometry and Dynamics' in Ohalo, Israel where this collaboration began. We also thank Tushar Das for taking the time to answer some questions.

\bibliographystyle{alpha}
\bibliography{main}

\end{document}

\end{document}